\documentclass{article}
\usepackage{geometry}
\usepackage{hyperref}
\usepackage{amssymb}
\usepackage{mathrsfs}
\usepackage{verbatim}
\usepackage{extarrows}
\usepackage{latexsym, cite, bm, amsthm, amsmath}
\usepackage{epsfig,graphicx, epsf,wrapfig}
\usepackage{epstopdf}
\usepackage{amsmath}\numberwithin{equation}{section}
\graphicspath{{figure/}}
\usepackage{subfigure}
\usepackage{tikz}
\usepackage{enumitem}
\usepackage{tabularx,multirow}
\usepackage{titletoc}
\usepackage{stmaryrd}
\usepackage{setspace} 
\usepackage[all,2cell]{xy}    
\UseAllTwocells               
\newcommand{\dif}{\mathop{}\mathrm{d}}
\newcommand{\n}{\mathop{}\mathrm{n}}
\newcommand{\bbR}{\mathbb{R}}

\newcommand{\bl}{\boldsymbol}
\newcommand{\T}{\mathcal{T}}
\newcommand{\E}{\mathcal{E}}

\newcommand{\wt}{\widetilde}
\newcommand{\lal}{\langle}
\newcommand{\ral}{\rangle}
\newcommand{\pt}{\partial}

\newtheorem{thm}{Theorem}[section]

\newtheorem{lem}{Lemma}[section]
\newtheorem{cro}{Corollary}[section]
\newtheorem{re}{Remark}[section]
\newcommand{\wuhao}{\fontsize{10.5pt}{12.6pt}\selectfont}        

\title{An arbitrary order mixed finite element method with boundary value correction for the Darcy flow on curved domains}
\date{}
\author{Yongli Hou and
Yanqiu Wang
\footnote{Corresponding author}}

\begin{document}

\maketitle

\wuhao
\begin{abstract}
We propose a boundary value correction method for the Brezzi-Douglas-Marini mixed finite element discretization of the Darcy flow with non-homogeneous Neumann boundary condition
on 2D curved domains. The discretization is defined on a body-fitted triangular mesh, i.e. the boundary nodes of the mesh lie on the curved physical boundary. However, the boundary edges of the triangular mesh, which are straight, may not coincide with the curved physical boundary.
A boundary value correction technique is then designed to transform the Neumann boundary condition from the physical boundary to the boundary of the triangular mesh. One advantage of the boundary value correction method is that it avoids using curved mesh elements and thus reduces the complexity of implementation. 
We prove that the proposed method reaches optimal convergence for arbitrary order discretizations.
Supporting numerical results are presented.

\end{abstract}

\vspace{0.9em}
\section{\textbf Introduction}

Many practical problems arising in science and engineering are posed on the curved domain $\Omega$.
Classical finite element methods defined on a polygonal approximation domain $\Omega_h$ often suffer from
an additional geometric error due to the difference between $\Omega$ and $\Omega_h$.
This leads to a loss of accuracy for higher-order discretizations \cite{Strang_Berger_The_change_in_solution_due_to_change_in_domain1973,Thomee_Polygonal_domain_approximation_in_Dirichlet's_problem1973}. Moreover, when the problem is equipped with an essential boundary condition, effective ways to transfer the boundary condition from $\partial\Omega$ to $\partial\Omega_h$ must be designed.
Different strategies have been proposed to solve this problem, for example,
the isoparametric finite element method \cite{isoparametric_1968,Optimal_isoparametric_finite_elements_1986} and
the isogeometric analysis \cite{Isogeometric_analysis_Hughes_2005,Isogeometric_Analysis_Cottrell_2009}.
Another strategy is the boundary value correction method \cite{1972},
which shifts the essential boundary condition from $\partial\Omega$ to $\partial\Omega_h$, and results in a modified variational formulation. One advantage of the boundary value correction method is that it avoids using curved mesh elements and thus reduces the complexity of implementation. 
Recently, there are many works utilizing the boundary correction strategy,
including the discontinuous Galerkin method via extensions from subdomains \cite{Cockburn_Boundary_conforming_DG_2009,Cockburn2012}, transferring technique based on the path integral \cite{Oyarzua_high_order_MFEM},
the shifted boundary method \cite{Part1,high_order_SBM}, the cutFEM \cite{Burman_boundary_value_correction_2018},
the boundary-corrected virtual element method \cite{High_order_VEM_2D3D} and the boundary-corrected weak Galerkin method \cite{Yiliu_WG}, etc.

For the mixed finite element method (MFEM), the Neumann boundary condition becomes essential. Again, higher-order MFEM suffers from a loss of accuracy on curved domains.
In \cite{Bertrand_2014_lowest,Bertrand_2014_High_order,Bertrand_2016_High_order}, the authors studied a parametric Raviart-Thomas MFEM on curved domains, which is a generalization of the isoparametric method to the MFEM.
In \cite{Puppi_cutFEM}, a cutFEM method is proposed in the curved domain, which is based on a primal mixed formulation \cite{book:FE_RT}.
The authors of \cite{Puppi_cutFEM} show that the cutFEM method reaches suboptimal convergence.

In this paper, we propose a new boundary corrected MFEM based on the primal mixed formulation \cite{book:FE_RT}, and prove that it has optimal convergence rate.
Similarly to \cite{fractual_non_matching_2012}, we weakly impose the Neumann boundary condition.
Then, a boundary value correction technique is designed to pull the Neumann boundary data from $\partial\Omega$ to $\partial\Omega_h$. However, unlike the boundary correction in \cite{1972,high_order_SBM,Part1}, which considers the Dirichlet boundary condition, the Neumann boundary condition involves the outward normal vector on the boundary, which is different in $\partial\Omega$ and $\partial\Omega_h$. This poses extra difficulty in the design and analysis of our scheme.
Finally, following \cite{Pei_Cao_XFEM_CMA}, we added a term $(\mathrm{div}\,\cdot,\mathrm{div}\,\cdot)$ to the discrete formulation to increase stability.

   The paper is organized as follows. In Section \ref{sec2:notation}, we introduce some notation and settings.
   In Section \ref{sec3:model_problem}, we describe the model problem and introduce the boundary value correction method.
   In Section \ref{sec4:discrete}, the discrete space and the discrete form are defined and analyzed.
   In Section \ref{sec5:error}, we prove the optimal error estimate.
   In Section \ref{sec7:numerical}, numerical results are presented. 
   Finally, we draw our conclusions in Section \ref{sec8:conclusion}.

\section{Notations and preliminaries}\label{sec2:notation}

Let $\Omega$ be a bounded open set in $\bbR^2$ with Lipschitz continuous and piecewise $C^2$ boundary $\Gamma$. Let $\T_h$ be a body-fitted triangulation of $\Omega$, i.e., all boundary vertices of $\T_h$ lie on $\Gamma$. We assume that $\T_h$ is geometrically conforming, shape regular, and quasi-uniform.
Denote by $\Omega_h$ the polygonal region occupied by the triangular mesh $\T_h$, and by $\Gamma_h$ the boundary of $\Omega_h$. When $\Omega$ has a curved boundary, the polygonal region $\Omega_h$ does not coincide with $\Omega$, as shown in Fig. \ref{Fig:tilde_K}. Denote by $\E_h^b$ the set of all boundary edges in $\T_h$ and by $\T_h^b$ all mesh elements that contain at least one edge in $\E_h^b$. For each $e\in\E_h^b$, denote by $\tilde e\subset \Gamma$ the curved edge cut by two end points of $e$ (see Fig. \ref{Fig:tilde_K}). We further assume that each $\tilde{e}$ is $C^2$ continuous, i.e. $\tilde{e}$ cannot cross the connection point of two $C^2$ continuous pieces of $\Gamma$. Note that $\tilde{e}$ may coincide with $e$ when part of $\Gamma$ is flat.

We use $\Omega_h^e$ to denote a crescent-shaped region surrounded by $e$ and $\tilde{e}$, as shown in Fig. \ref{Fig:tilde_K}.
There are three possibilities: (a) $\Omega_h^e\subset \Omega\backslash \Omega_h$; (b) $\Omega_h^e\subset \Omega_h\backslash \Omega$; (c) $e=\tilde{e}$ and $\Omega_h^e = \emptyset$.
If $\Omega$ is convex, we have $\Omega_h\subset \Omega$.
But when $\Omega$ is not convex, there is no inclusion relationship between $\Omega$ and $\Omega_h$. 

\begin{figure}[!htbp]
  \centering
  \subfigure[]
  {
  \begin{tikzpicture}
  \draw[red,very thick] (-1.3,0)--(1,0);
  \draw[thick] (-1.3,0)--(-0.3,-2);
  \draw[thick] (1,0)--(-0.3,-2);
  \draw[blue,thick] (1,0) arc (40:140:1.5);
  \draw (-0.2,0.5) node [below] {$\Omega_h^e$};
  \draw[red] (0.6,0) node [below] {$e$};
  \draw[blue] (0.2,0.6) node [right] {$\wt e$};
  \end{tikzpicture}
  }\qquad\qquad
  \subfigure[]
  {
  \begin{tikzpicture}
  \draw[red,very thick] (-1.3,0)--(1,0);
  \draw[thick] (-1.3,0)--(-0.3,-2);
  \draw[thick] (1,0)--(-0.3,-2);
  \draw[blue,thick] (-1.3,0) arc (220:319:1.5);
  \draw (-0.2,0) node [below] {$\Omega_h^e$};
  \draw[red] (0.6,0) node [above] {$e$};
  \draw[blue] (0.2,-0.6) node [right] {$\wt e$};
  \end{tikzpicture}
  }\qquad\qquad
  \subfigure[]
  {
  \begin{tikzpicture}
  \draw[red,very thick] (-1.3,0)--(1,0);
  \draw[blue,,dotted,very thick] (-1.3,0)--(1,0);
  \draw[thick] (-1.3,0)--(-0.3,-2);
  \draw[thick] (1,0)--(-0.3,-2);
  \draw[] (0.0,0) node [above] {$ e=\wt{e}$};
  \end{tikzpicture}
  }
  \caption{(a) The case when $\Omega_h^e\subset \Omega\backslash \Omega_h$. 
  (b) The case when $\Omega_h^e\subset \Omega_h\backslash \Omega$.
  (c) The case when $e=\wt e$ and $\Omega_h^e = \emptyset$.}
\label{Fig:tilde_K}
\end{figure}
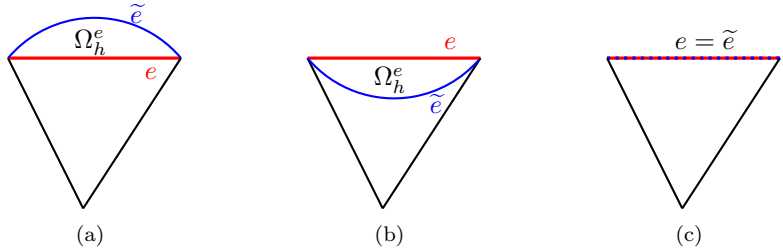

Throughout the paper, we assume that the mesh $\T_h$ is geometrically conforming, shape regular, and quasi-uniform \cite{FEM_Brenner}.
For each triangle $K\in{\T}_h$, denote by $h_K$ the diameter of $K$. Let $h=\max_{K\in{\T}_h}h_K$. Similarly, for each edge $e\in \E_h^b$, denote by $h_e$ the length of $e$. 
Denote by $P_k(K)$ the space of polynomials on $K\in\T_h$
with degree less than or equal to $k$.
We use the standard notation for Sobolev spaces \cite{FEM_Brenner}. 
Let $H^m(S)$, for $m\in\bbR$ and $S\subset\bbR^2$ be the usual Sobolev space equipped with the norm $\|\cdot\|_{m,S}$ and the seminorm $|\cdot|_{m,S}$.
When $m=0$, the space $H^0(S)$ coincides with the square integrable space $L^2(S)$.
Denote by $L_0^2(S)$ the subspace of $L^2(S)$ consisting of functions with mean value zero.
The above notation extends to a portion $s\subset\Gamma$ or $s\subset\Gamma_h$.
For example, $\|\cdot\|_{m,s}$ is the Sobolev norm on $H^m(s)$. 
We also use the notation $(\cdot,\cdot)_S$ to indicate the $L^2$ inner-product on $S\subset\bbR^2$,
and $\lal\cdot,\cdot\ral_s$ to indicate the duality pair on $s\subset\Gamma$ or $s\subset\Gamma_h$.

All of the above notations extend to vector-valued functions, following the usual rule of product spaces. We use bold face letters to distinguish the vector-valued function spaces from the scalar-valued ones. For example, $\bl P_k(K) = [P_k(K)]^2$ and $\bl H^m(S) = [H^m(S)]^2$ are the spaces of vector-valued functions with each component in $P_k(K)$ and $H^m(S)$, respectively. 
Define
\begin{equation*}
\begin{aligned}
H(\mathrm{div},S)&=\{\bl v\in \bl L^2(S):\mathrm{div}\,\bl v \in L^2(S)\},\\
H^m(\mathrm{div},S)&=\{\bl v\in \bl H^m(S):\mathrm{div}\,\bl v \in H^m(S)\}.\\
\end{aligned}
\end{equation*}
The norm in $H(\mathrm{div},S)$ is defined by
$$\| \bl v\|_{H(\mathrm{div},S)}=(\| \bl v\|_{0,S}^2+\|\mathrm{div}\, \bl v\|_{0,S}^2)^{1/2}.$$

Throughout the paper, we use $\lesssim$ to denote less than or equal up to a constant,
and the analogous notation $\gtrsim$ to denote greater than or equal up to a constant.

Finally, we introduce some well-known inequalities. 
\begin{lem}\label{lem:TRACE}
(Trace Inequality \cite{FEM_Brenner}). For $K\in\T_h$, we have
\begin{equation*}
\begin{aligned}
\|v\|_{0,\pt K}{\lesssim}~h_K^{-1/2}\| v\|_{0,K}+h_K^{1/2}|v|_{1,K},\qquad \forall v\in H^1(K).\\
\end{aligned}
\end{equation*}
\end{lem}

\begin{lem}\label{lem:inverse}
(Inverse Inequality \cite{FEM_Brenner}). Given integers $0\le m\le l$. For $K\in\T_h$, we have
\begin{equation*}
\begin{aligned}
|q|_{l, K}{\lesssim}~h_K^{m-l}| q|_{m,K},\qquad \forall q\in P_l(K).
\end{aligned}
\end{equation*}
\end{lem}

\begin{lem}\label{lem:trace_inverse}
(Trace-inverse Inequality \cite{FEM_Brenner}). Given an integer $l\ge 0$. For $K\in\T_h$ and $e$ be an edge of $K$, we have
\begin{equation*}
\begin{aligned}
\|q\|_{0,e}{\lesssim}~h_K^{-1/2}\| q\|_{0, K},\qquad \forall q\in P_l(K).
\end{aligned}
\end{equation*}
\end{lem}

\begin{lem}\label{lem:BH}
(Bramble-Hilbert Lemma \cite{FEM_Brenner}). Given an integer $k\ge 0$. For $K\in\T_h$ and integers $0\le m\leq k+1$, one has
\begin{equation*}
\begin{aligned}
\inf_{\xi\in P_k(K)}\bigg(\sum_{i=0}^m h_K^i|v-\xi|_{i,K}\bigg)\lesssim h_K^m|v|_{m,K},\qquad \forall v\in H^m(K).
\end{aligned}
\end{equation*}
\end{lem}

\section{ Model problem and boundary value correction method}\label{sec3:model_problem}
In this section, we briefly describe the model problem and the boundary value correction method \cite{1972}.
Consider the Darcy's equation with the Neumann boundary condition
\begin{equation}
\begin{aligned}
\bl u + \nabla p =&~0 ,\quad &\text{in} \,\Omega ,&\\ \label{primal_problem}
\mathrm{div}\,\bl u=&~f ,\quad &\text{in} \, \Omega,\\
\bl u\cdot\bl\n=&~g_N ,\quad &\text{on} \, \Gamma,
\end{aligned}
\end{equation}
with $f\in L^2(\Omega)$ and $g_N\in H^{-1/2}(\Gamma)$.
Here, $\bl{\n}$ denotes the unit outward normal on $\Gamma$.
Problem (\ref{primal_problem}) admits a unique weak solution $(\bl u,p)\in H(\mathrm{div},\Omega)\times L_0^2(\Omega)$ as long as the following compatibility condition holds:
\begin{align}\label{g_comp_cond}
    \int_\Omega f\dif x=\int_\Gamma g_N\dif s.
\end{align}

The weak formulation of Problem (\ref{primal_problem}) can be written as: Find $(\bl u,p)\in H(\mathrm{div},\Omega)\times L_0^2(\Omega)$ satisfying $\bl u\cdot\bl\n=g_N$ on $\Gamma$, such that
\begin{equation}
\left\{
\begin{aligned}\label{Weak_P}
(\bl u,\bl v)_{\Omega}-(\mathrm{div}\,\bl v,p)_{\Omega}&=0, \quad &\forall&\bl v\in H_0(\mathrm{div},\Omega),\\
-(\mathrm{div}\,\bl u,q)_{\Omega}&=-(f,q)_{\Omega},  \quad &\forall &q\in L_0^2(\Omega),
\end{aligned}
\right.
\end{equation}
where $H_0(\mathrm{div},\Omega) = \{\bl v\in H(\mathrm{div},\Omega): \bl v\cdot \bl n = 0 \textrm{ on }\Gamma\}$.
The existence and uniqueness of a weak solution to the mixed problem (\ref{Weak_P}) can be found in \cite{FEM_Brenner}.

We use the Brezzi-Douglas-Marini (BDM) finite element method \cite{BDM_1985} on triangulation $\T_h$ to discretize Equation (\ref{Weak_P}). The discretization is defined on $\Omega_h$ rather than on $\Omega$. 
Note that in \eqref{Weak_P}, the Neumann boundary condition $\bl u\cdot\bl\n=g_N$ on $\Gamma$ becomes essential. This causes the main difficulty in its finite element discretization when $\Gamma$ is curved and hence not equal to $\Gamma_h$. Another difficulty is that the compatibility condition \eqref{g_comp_cond} holds only on $\Omega$ but not on $\Omega_h$. We shall take care of these issues one by one.
First, we use a boundary correction technique \cite{1972} to transform the boundary data from $\Gamma$ to $\Gamma_h$.

Assume that there exists a map $\rho_h:\Gamma_h\rightarrow\Gamma$ such that
\begin{equation}
\begin{aligned}\label{map}
 \rho_h(\bl x_h):=\bl x_h+\delta_h(\bl x_h)\bl \nu_h(\bl x_h),
 \end{aligned}
\end{equation}
as shown in Fig. \ref{Fig:Mh}, where $\bl\nu_h$ is the unit vector pointing from $\bl x_h \in \Gamma_h$ to $\rho_h(\bl x_h) \in \Gamma$, and $\delta_h(\bl x_h)=|\rho_h(\bl x_h)-\bl x_h|$.
For convenience, we denote $\bl x:=\rho_h(\bl x_h)$. Define $\wt{\bl\n}(\bl x_h) := \bl\n \circ  \rho_h(\bl x_h)$, i.e. the unit outward normal vector on $\bl x \in \Gamma$ is pulled to the corresponding $\bl x_h$. Denote by $\bl{\n}_h$ the unit outward normal vector on $\Gamma_h$. Since $\Gamma$ is piecewise $C^2$ continuous, there exists a map $\rho_h$ satisfying  \cite{Burman_boundary_value_correction_2018}
\begin{equation}
\begin{aligned}\label{deta}
\delta=\sup_{\bl x_h\in\Gamma_h}\delta_h(\bl x_h)\lesssim h^2,\qquad \|\wt{\bl\n}-\bl\n_h\|_{L^\infty(\Gamma_h)}\lesssim h.
\end{aligned}
\end{equation}

\begin{re}
In this paper, we do not specify the map $\rho_h$. We only require that the distance function $\delta_h(\bl x_h)$ satisfies (\ref{deta}).
For body-fitted meshes, this surely holds by simply setting $\bl \nu_h = \bl \n_h$. In fact, (\ref{deta}) can be true on certain unfitted meshes \cite{Burman_boundary_value_correction_2018}. The method and analysis in this paper extend to unfitted meshes that satisfy (\ref{deta}), with a little modification.
\end{re}
\begin{figure}[!htbp]
\centering 
{
\begin{tikzpicture}
\draw[red,very thick] (-1.71,0.8)--(1.69,0.8);
\draw[blue, thick] (1.8,0.6) arc (25:155:2);
\draw[gray, thick,->] (-0.3,0.8)--(-0.6,1.705);
\draw[blue, thick,->](-0.6,1.705)--(-0.85,2.5);
\draw (-0.1,0.7) node [below] {${\bm x_h}$};
\draw[red] (1,0.4) node [right] {$\Gamma_h$};
\draw[blue] (1.8,0.6) node[right] {$\Gamma$};
\draw (0.15,1.2) node {$\delta_h(\bm x_h)$};
\draw (-0.7,1.6) node[below] {$\bm x$};
\draw (-0.2,2.2) node[left] {$\bl \nu_h$};
\filldraw [black](-0.3,0.8) circle [radius=2pt];
\filldraw [black](-0.6,1.67) circle [radius=2pt];
\end{tikzpicture}
}\qquad\quad
 \caption{The distance $\delta_h(\bl x_h)$ and the unit vector $\bl{\nu_h}$ on $\Gamma_h$. 
}
\label{Fig:Mh}
\end{figure}
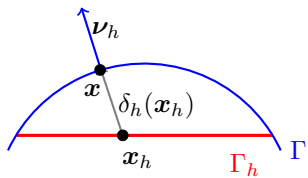

For a sufficiently smooth vector-valued function $\bl v$ defined between $\Gamma$ and $\Gamma_h$,
consider its $m$-th order Taylor expansion along  $\bl{\nu_h}$
\begin{equation*}
\begin{aligned}
\bl v(\rho_h(\bl x_h))&=\sum_{j=0}^{m}\frac{\delta_h^j(\bl x_h)}{j!}\partial_{\bl\nu_h}^j \bl v(\bl x_h)+R^m\bl v(\bl x_h)\\
&=:(T^m\bl v)(\bl x_h)+R^m\bl v(\bl x_h),
\end{aligned}
\end{equation*}
where $\partial_{\bl\nu_h}^j$ is the $j$-th order directional derivative and the remainder term satisfies
$$|R^m\bl v(\bl x_h)|=o(\delta^m).$$

For simplicity of notation, denote
\begin{equation}
\begin{aligned}\label{T1}
(T^m\bl v)(\bl x_h):=\bl v(\bl x_h)+(T_1^m\bl v)(\bl x_h) \quad
\textrm{where} \quad
(T_1^m\bl v)(\bl x_h):=\sum_{j=1}^{m}\frac{\delta_h^j(\bl x_h)}{j!}\partial_{\bl\nu_h}^j \bl v(\bl x_h).
\end{aligned}
\end{equation}
Both $T^m \bl v$ and $T_1^m \bl v$ are functions defined on $\Gamma_h$.
Denote by $\wt{\bl v}(\bl x_h):=\bl v\circ \rho_h(\bl x_h)$ the pullback of $\bl v$ from $\Gamma$ to $\Gamma_h$. Then clearly
\begin{equation}
\begin{aligned}\label{T-R}
\wt{\bl v} - T^m\bl v = R^m \bl v \qquad \textrm{on }\Gamma_h.
\end{aligned}
\end{equation}




Assume that (\ref{deta}) holds, we have the following two lemmas.

\begin{lem}\label{lem:Omega_he2K}
(cf. \cite{Yiliu_WG}). Given an integer $j\ge 0$. For $K\in \T_h^b$ and $q\in P_j(K)$, one has
\begin{equation*}
\begin{aligned}
\sum_{e\subset\pt K\cap\Gamma_h}\|q\|_{0,\Omega_h^e}^2\lesssim h_K\|q\|_{0,K}^2.
\end{aligned}
\end{equation*}
\end{lem}

\begin{lem}\label{lem:B_K1994}
(cf. \cite{Bramble_King1994}). For each $e\in\E_h^b$ and $v\in H^1(\Omega \cup \Omega_h)$, one has 
\begin{align}
\|v\|_{0,e}&\lesssim \|v\|_{0,\wt e}+ h|v|_{1,\Omega_h^e}, \label{BK_eh2e}\\
\|v\|_{0,\wt e}&\lesssim\|v\|_{0,e}+h|v|_{1,\Omega_h^e},\label{BK_e2eh}\\
\|v\|_{0,\Omega_h^e}&\lesssim \delta^{1/2}\|v\|_{0,\wt e}+\delta\|\nabla v\|_{0,\Omega_h^e}.\label{BK_Me2Me}
\end{align}
Moreover, when $v|_{\pt \Omega}=0$, one has 
\begin{equation}
\begin{aligned}\label{BK_e2Omeg_e}
\|v\|_{0,e}\lesssim h\|v\|_{1,\Omega_h^e}.
\end{aligned}
\end{equation}
\end{lem}
Lemmas \ref{lem:Omega_he2K}-\ref{lem:B_K1994}
can be easily extended to vector-valued functions.

\section{The finite element discretization }\label{sec4:discrete}

For a given integer $k\geq 1$, define the BDM finite element space on $\T_h$ by
 \begin{equation*}
\begin{aligned}
V_h&=\{\bl v_h\in H(\mathrm{div},\Omega_h): \bl v_h|_K\in \bl P_k(K),\, \forall \, K\in\T_h\},\\
Q_h&=\{q_h\in L^2(\Omega_h): q_h|_K\in P_{k-1}(K),\, \forall \, K\in\T_h\},\\
\end{aligned}
\end{equation*}
and denote $Q_{0h}:=Q_{h}\cap L_0^2(\Omega_h)$.

For $K\in\T_h$, define the local interpolation $I_K:\bl H^s(K)\rightarrow BDM_k(K), s>1/2$, by
\begin{equation}
\begin{cases}\label{Pi_k}
(I_K \bl v,\nabla\psi_{k-1})_K =(\bl v,\nabla\psi_{k-1})_K,\quad &\forall\,\psi_{k-1}\in  P_{k-1}(K),\\
\lal I_K \bl v\cdot\bl\n_h,\psi_k \ral_e =\lal  \bl v\cdot\bl\n_h,\psi_k \ral_e,\quad &\forall\,\psi_k\in P_k(e), \text{ on all edge }e\subset\pt K,\\
( I_K \bl v,\text{curl}(b_K\psi_{k-2}))_K=(\bl v,\text{curl}(b_K\psi_{k-2}))_K,\quad &\forall\,\psi_{k-2}\in  P_{k-2}(K),
\textrm{ when }k\ge 2,
\end{cases}
\end{equation}
where $b_K=\lambda_1\lambda_2\lambda_3$ is a bubble function, and $\lambda_i,\,i=1,2,3$, are the barycentric coordinates associated with $K$.
Denote the global interpolation by $I_h=\prod_{K\in\T_h}I_K : H^s(\Omega_h)\to V_h$.

Denote by $\Pi_{k-1}^{0,K}$ and $\Pi_k^{0,e}$ the $L^2$ orthogonal projections onto $P_{k-1}(K)$ and $P_k(e)$, respectively.
Define the global projection $\Pi_h=\left(\prod_{K\in\T_h}\Pi_{k-1}^{0,K}\right): L^2(\Omega_h)\to Q_h$.

\begin{lem}
(cf. \cite{BDM_1985}). The following commutative diagram holds:
\[
\xymatrix@R=2em@C=4em{
V \ar[r]^{\mathrm{div}} \ar[d]^{I_h}  & L^2(\Omega_h)\ar[d]^{\Pi_h}\\
V_h \ar[r]^{\mathrm{div}}  & Q_h
}
\]
where $V:=H(\mathrm{div},\Omega_h)\cap H^s(\Omega_h)$ with $s>1/2$. That is, for $\bl v\in V$ one has
\begin{equation}
\begin{aligned}\label{comm_diag_property}
\mathrm{div}\,I_K\bl v=\Pi_{k-1}^{0,K}\mathrm{div}\,\bl v
\qquad \forall K\in \T_h.
\end{aligned}
\end{equation}
\end{lem}

\begin{lem}\label{lem:Ih_sta}
  For $K\in\T_h$, the interpolation $I_K$ satisfies, for $\bl v\in \bl H^1(K)$,
  \begin{align}
  \|I_K\bl v\|_{0,K}&\lesssim \|\bl v\|_{0,K}+h_K|\bl v|_{1,K},\label{Ih_sta_0norm}\\
  |I_K\bl v|_{1,K}&\lesssim |\bl v|_{1,K}.\label{Ih_sta_1norm}
  \end{align}
  \end{lem}
  \begin{proof}
  Inequality (\ref{Ih_sta_0norm}) is well known \cite{MFEM}. 
  Inequality \eqref{Ih_sta_1norm} follows immediately from applying (\ref{Ih_sta_0norm}) to $\bl v-\frac{1}{|K|}\int_K \bl v\dif x$, the inverse inequality, and the Bramble-Hilbert lemma. More specifically, denoting 
 $\bar{\bl v}=\frac{1}{|K|}\int_K \bl v\dif x$, one has
  \begin{equation*}
  \begin{aligned}
  |I_K\bl v|_{1,K}&=|I_K\bl v-\bar{\bl v}|_{1,K}=|I_K(\bl v-\bar{\bl v})|_{1,K}\\
  &\lesssim h_K^{-1}\|I_K(\bl v-\bar{\bl v})\|_{0,K}\\
  &\lesssim h_K^{-1}(\|\bl v-\bar{\bl v}\|_{0,K}+h_K|\bl v|_{1,K})\\
  &\lesssim |\bl v|_{1,K}.
  \end{aligned}
  \end{equation*}
  \end{proof}

  By the lemmas \ref{lem:TRACE},\ref{lem:inverse},\ref{lem:BH}, it is not hard to see that
\begin{lem}\label{lem:Pi_Ih_app}
Given a positive integer $k$. For any nonnegative integers $j\leq i\leq k+1$ and any $K\in\T_h$, one has
\begin{align}
 h_K^j|v-\Pi_k^{0,K}v|_{j,K}&\lesssim h_K^{i-j}|v|_{i,K}, &&\forall v\in H^i(K),\label{pi_app}\\
 h_K^j|\bl v-I_K \bl v|_{j,K}&\lesssim h_K^{i-j}|\bl v|_{i,K}, &&\forall \bl v\in \bl H^i(K),\label{Ih_app}\\
 \|\mathrm{div}\,(\bl v-I_K\bl v)\|_{0,K}&\lesssim h_K^i|\mathrm{div}\,\bl v|_{i,K}, &&\forall \bl v\in H^i(\mathrm{div},K).\label{Ih_div_app}
 \end{align}
\end{lem}

Finally, we recall the extension property of Sobolev spaces:
\begin{lem}\label{lem:extension} (cf. \cite{book:Extension}).
Given a Lipschitz domain $\Omega$ in $\bbR^2$ and $s\in\bbR,\, s\geq 0$, there exists an extension operator $E: H^s(\Omega)\to H^s(\bbR^2)$ such that
  \begin{align*}
    E v|_{\Omega}= v,\quad \|E v\|_{s,\bbR^2}\lesssim \|v\|_{s,\Omega},\qquad \forall v\in H^s(\Omega).
  \end{align*}
  \end{lem}
  
  In Lemma \ref{lem:extension}, the hidden constant in $\lesssim$ may depend on the shape of $\Omega$. Throughout the paper, we shall only use Lemma \ref{lem:extension} on $\Omega$ but not on $\Omega_h$, in order to ensure that the hidden constant does not depend on $h$.
  For brevity, we also denote the extension by $v^E=E v$.
  The extension and its notation can be extended to vector-valued functions.
 
\subsection{Discretization with boundary correction}

Let $(\bl u,\, p)$ be the solution to (\ref{primal_problem}) and 
$(\bl u^E,\, p^E)$ be their extension to $\bbR^2$.
Testing with $\bl v_h\in V_h$ in $\Omega_h$ gives
\begin{equation}
\begin{aligned}\label{eq:deduce_w1}
(\bl u^E+\nabla p^E,\bl v_h)_{\Omega_h} = (\bl u^E,\bl v_h)_{\Omega_h}-(p^E,\mathrm{div}\,\bl v_h)_{\Omega_h}+\sum_{e\in\E_h^b}\lal\bl v_h\cdot\bl\n_h,p^E\ral_e.
\end{aligned}
\end{equation}
We point out that although $\bl u^E+\nabla p^E=0$ in $\Omega$, the term $\bl u^E+\nabla p^E$ is not in general equal to zero in $\Omega_h\backslash \Omega$.
Next, using (\ref{T-R}), it holds
\begin{equation*}
\begin{aligned}\label{Equality:Rm}
(T^m\bl u^E+R^m \bl u^E)\cdot\wt{\bl\n}-\wt{g}_N=0,\quad\text{on}\,\Gamma_h,
\end{aligned}
\end{equation*}
where $\wt{g}_N(\bl x_h) = g_N(\rho(\bl x_h))$ is a pull-back of the Neumann boundary data $g_N$ from $\Gamma$ to $\Gamma_h$.
Equation (\ref{eq:deduce_w1}) can then be rewritten into
\begin{equation}
\begin{aligned}\label{eq:weak_1st}
(\bl u^E,\bl v_h)_{\Omega_h}&+(\mathrm{div}\, \bl u^E, \mathrm{div}\, \bl v_h)_{\Omega_h}+\sum_{e\in\E_h^b}\lal h_K^{-1}T^m\bl u^E\cdot\wt{\bl\n}, T^m\bl v_h\cdot\wt{\bl\n}\ral_e-(p^E,\mathrm{div}\,\bl v_h)_{\Omega_h}\\
&+\sum_{e\in\E_h^b}\lal\bl v_h\cdot\bl\n_h,p^E\ral_e=(\bl u^E+\nabla p^E,\bl v_h)_{\Omega_h}+(f^E+(\mathrm{div}\, \bl u^E-f^E), \mathrm{div}\, \bl v_h)_{\Omega_h}\\
&\qquad\qquad\qquad\qquad\qquad\,+\sum_{e\in\E_h^b}\lal h_K^{-1}(\wt{g}_N-R^m\bl u^E\cdot\wt{\bl\n}), T^m\bl v_h\cdot\wt{\bl\n}\ral_e.
\end{aligned}
\end{equation}
Again, note that $\mathrm{div}\, \bl u^E-f^E$ vanishes in $\Omega$ but not necessarily vanishes in $\Omega_h\backslash \Omega$.

Similarly, for any $q_h\in Q_h$, one has
\begin{equation}
\begin{aligned}\label{weak_2nd}
(\mathrm{div}\, \bl u^E,q_h)_{\Omega_h}=(f^E+(\mathrm{div}\, \bl u^E-f^E),q_h)_{\Omega_h}.
\end{aligned}
\end{equation}

We then introduce a discretization for (\ref{primal_problem}), using \eqref{eq:weak_1st}-\eqref{weak_2nd}.
First, for $\bl w,\,\bl v \in  H(\mathrm{div},\Omega\cup\Omega_h)$ that are smooth enough so that $T^m$ is well-defined, and for $q\in  L^2(\Omega\cup\Omega_h)$, define bilinear forms
and linear forms
\begin{equation}
\begin{aligned}\label{bilinear forms}
a_h(\bl w,\bl v):&=(\bl w,\bl v)_{\Omega_h}+(\mathrm{div}\, \bl w, \mathrm{div}\, \bl v)_{\Omega_h}+\sum_{e\in\E_h^b}\lal h_K^{-1}T^m\bl w\cdot\wt{\bl\n}, T^m\bl v\cdot\wt{\bl\n}\ral_e,\\
b_{h1}(\bl v,q):&=-(\mathrm{div}\, \bl v,q)_{\Omega_h}+\sum_{e\in\E_h^b} \lal\bl v\cdot\bl\n_h,q\ral_e,\\
b_{h0}(\bl v,q):&=-(\mathrm{div}\,\bl v,q)_{\Omega_h},\\
l_h(\bl v):&=(f^E, \mathrm{div}\, \bl v)_{\Omega_h}+\sum_{e\in\E_h^b}\lal h_K^{-1}\wt{g}_N, T^m\bl v\cdot\wt{\bl\n}\ral_e.
\end{aligned}
\end{equation}
Define the discrete formulation: {\it Find $(\bl u_h,p_h)\in V_h\times Q_{0h}$ such that}
\begin{equation}
\left\{
\begin{aligned}\label{Wh1}
a_h(\bl u_h,\bl v_h)+b_{h1}(\bl v_h,p_h)&=l_h(\bl v_h), \quad &\forall&\bl v_h\in V_h,\\
b_{h0}(\bl u_h,q_h)&=-(f^E,q_h)_{\Omega_h},  \quad &\forall &q_h\in Q_{0h}.
\end{aligned}
\right.
\end{equation}
In the rest of this section, we shall prove the well-posedness of \eqref{Wh1}, and discuss its ``compatibility" condition.

\subsection{Well-posedness of (\ref{Wh1})} \label{subsec:wp}
We first define a mesh-dependent norm in $V_h$ as follows
\begin{equation*}
\begin{aligned}
\|\bl v\|_{0,h}&=\left(\|\bl v\|_{H(\mathrm{div},\Omega_h)}^2+\sum_{e\in\E_h^b}\|h_K^{-1/2}T^m\bl v\cdot\wt{\bl\n}\|_{0,e}^2\right)^{1/2}.
\end{aligned}
\end{equation*}
 Note that the norm $\|\cdot\|_{0,h}$ is also meaningful for functions in $H^m(\Omega\cup\Omega_h)$ smooth enough so that $T^m$ is well-defined.

\begin{lem}\label{lem:T1_bounded}
Under the assumption (\ref{deta}), for $\bl v_h\in V_h$ we have
\begin{align}
  \|h_K^{-1/2} T_1^m\bl v_h\|_{0,e}&\leq  h|\bl v_h|_{1,K},\label{T0_bounded}\\
\|h_K^{-1/2} T_1^m\bl v_h\|_{0,e}&\lesssim \|\bl v_h\|_{0,K}. \label{T1_bound}
\end{align}
\end{lem}

\begin{proof}
By the definition of $T_1^m\bl v_h$ in (\ref{T1}), and lemmas \ref{lem:TRACE}-\ref{lem:inverse}, we have
\begin{equation*}
\begin{aligned}
 \|h_K^{-1/2} T_1^m\bl v_h\|_{0,e}&
  =h_K^{-1/2}\left\|\sum_{j=1}^{m}\frac{\delta_h^j}{j!}\partial_{\bl\nu_h}^j  \bl v_h\right\|_{0,e}\\
  &\lesssim \sum_{j=1}^{m} h^{-1/2} \delta_h^{j}\|\partial_{\bl\nu_h}^j  \bl v_h\|_{0,e}\\
  &\lesssim \sum_{j=1}^{m} \left(\frac{ \delta_h}{h}\right)^{j}|\bl v_h|_{1,K}\\
& \lesssim \left(\frac{ \delta_h}{h}\right)|\bl v_h|_{1,K}\\
& \lesssim h|\bl v_h|_{1,K},
\end{aligned}
\end{equation*}
where we have used (\ref{deta}) in the last inequality. This completes the proof of \eqref{T0_bounded}.
Inequality (\ref{T1_bound}) follows immediately from the inverse inequality in Lemma \ref{lem:inverse}.
\end{proof}

\begin{lem}\label{lem:vn_e2h}
  Under the assumption (\ref{deta}), for $\bl v_h\in V_h$, we have
\begin{equation*}
\begin{aligned}
\sum_{e\in\E_h^b}h_K^{-1}\|\bl v_h\cdot\bl\n_h\|_{0,e}^2\lesssim \|\bl v_h\|_{0,h}^2.
\end{aligned}
\end{equation*}
\end{lem}

\begin{proof}
For any $e\in\E_h^b$, by the definition of $T^m$ and lemmas \ref{lem:T1_bounded}, \ref{lem:TRACE}, and \ref{lem:inverse}, one gets
\begin{equation*}
\begin{aligned}
\|\bl v_h\cdot\bl\n_h\|_{0,e}&=\|\bl v_h\cdot\wt{\bl\n}+\bl v_h\cdot(\wt{\bl\n}-\bl\n_h)\|_{0,e} \\
&=\|T^m\bl v_h\cdot\wt{\bl\n}-T_1^m\bl v_h\cdot\wt{\bl\n}+\bl v_h\cdot(\wt{\bl\n}-\bl\n_h)\|_{0,e} \\
&\leq \|T^m\bl v_h\cdot\wt{\bl\n}\|_{0,e}+\|T_1^m\bl v_h\cdot\wt{\bl\n}\|_{0,e}+\|\bl v_h\cdot(\wt{\bl\n}-\bl\n_h)\|_{0,e} \\
&\lesssim \|T^m\bl v_h\cdot\wt{\bl\n}\|_{0,e}+h_K^{1/2}\|\bl v_h\|_{0,K}.
\end{aligned}
\end{equation*}
Summing up over $e\in\E_h^b$, this completes the proof of the lemma.
\end{proof}

\begin{lem}\label{lem:bounded}
For $\bl u_h$, $\bl v_h\in V_h$ and $q_h\in Q_{0h}$, we have
\begin{equation*}
\begin{aligned}
a_h(\bl u_h,\bl v_h)&\leq \|\bl u_h\|_{0,h} \|\bl v_h\|_{0,h},\qquad &a_h(\bl v_h,\bl v_h)&=  \|\bl v_h\|_{0,h}^2,\\
b_{h1}(\bl v_h,q_h)& \lesssim \|\bl v_h\|_{0,h} \|q_h\|_{0,\Omega_h},\qquad &b_{h0}(\bl v_h,q_h)&\lesssim\|\bl v_h\|_{0,h} \|q_h\|_{0,\Omega_h}.\\
\end{aligned}
\end{equation*}
\end{lem}

\begin{proof}
From the definition, it is obvious that $a_h(\bl v_h,\bl v_h)=  \|\bl v_h\|_{0,h}^2$. The upper bounds of $a_h$ and $b_{h0}$ follow immediately from the Schwarz inequality.

Using the Schwarz inequality and lemmas \ref{lem:vn_e2h}, \ref{lem:trace_inverse}, we obtain
\begin{equation*}
\begin{aligned}
|b_{h1}(\bl v_h,q_h)|&\leq \|\mathrm{div}\,\bl v_h\|_{0,\Omega_h}\|q_h\|_{0,\Omega_h}+\sum_{e\in\E_h^b} h_K^{-1/2}\|\bl v_h\cdot\bl\n_h\|_{0,e}h_K^{1/2}\|q_h\|_{0,e}\\
&\lesssim \|\mathrm{div}\,\bl v_h\|_{0,\Omega_h}\|q_h\|_{0,\Omega_h}+\|\bl v_h\|_{0,h}\|q_h\|_{0,\Omega_h}\\
&\lesssim \|\bl v_h\|_{0,h}\|q_h\|_{0,\Omega_h}.
\end{aligned}
\end{equation*}
\end{proof}

Lemma \ref{lem:bounded} gives the boundedness of $a_h(\cdot,\cdot)$, $b_{h1}(\cdot,\cdot)$, $b_{h0}(\cdot,\cdot)$, and the coercivity of $a_h(\cdot,\cdot)$.
We then prove the $\inf$-$\sup$ condition for $b_{h1}(\cdot,\cdot)$ and $b_{h0}(\cdot,\cdot)$.
Note that a function $q_h\in Q_{0h}$ is mean-value free on $\Omega_h$ but not on $\Omega$. This poses additional difficulty in the analysis.
To this end, we first introduce the following notation to distinguish between two types of crescent-shaped regions $\Omega_h^{e}$, for each $e\in \E_h^b$, as shown in Fig. \ref{Fig:yueya_domain}, 
\[\label{Ome_he}
\Omega_h^{e}=
\begin{cases}
\Omega_h^{e,+}, \quad\text{in} \,\Omega\backslash \Omega_h,\\
\Omega_h^{e,-}, \quad\text{in} \,\Omega_h\backslash \Omega.
\end{cases}\qquad
\sigma=
\begin{cases}
  1, \quad&\text{in}\, \Omega_h^{e,+}, \\
  -1, \quad&\text{in} \,\Omega_h^{e,-}.
  \end{cases}
\]
\begin{figure}[!htbp]
  \centering
  {
  \includegraphics[height=4cm,width=6cm]{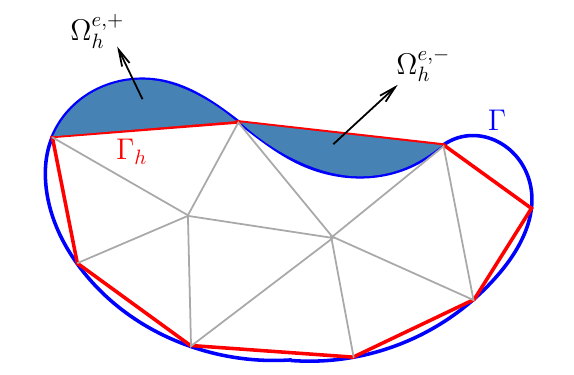}
  }
  \caption{The true boundary $\Gamma$ (blue curve), the approximated boundary $\Gamma_h$ (red lines) and two typical regions $\Omega_h^{e,+},\Omega_h^{e,-}$ bounded by $\Gamma$ and $\Gamma_h$.}
  \label{Fig:yueya_domain}
  \end{figure}

\begin{lem}\label{lem:LBB}
  (Inf-Sup condition). When $h$ is sufficiently small, there exist positive constants $\beta_0$ and $\beta_1$ independent of $h$ such that, for $q_h\in Q_{0h}$, 
\begin{equation*}
\begin{aligned}\label{LBB12}
\sup_{\bl v_h\in V_h}\frac{b_{h1}(\bl v_h,q_h)}{\|\bl v_h\|_{0,h}}\geq \beta_1 \|q_h\|_{0,\Omega_h},\qquad
\sup_{\bl v_h\in V_h}\frac{b_{h0}(\bl v_h,q_h)}{\|\bl v_h\|_{0,h}}\geq \beta_0 \|q_h\|_{0,\Omega_h}.
\end{aligned}
\end{equation*}
\end{lem}

\begin{proof}
Each function $q_h\in Q_{0h}\subset L_0^2(\Omega_h)$ can be naturally extended
to $\Omega\cup \Omega_h$, by filling the gap region $\Omega\backslash \Omega_h$ with the same polynomial on neighboring mesh
element. For simplicity, we still denote the extension by $q_h$.
Define $\bar q_h=q_h-\frac{1}{|\Omega|}\int_{\Omega}q_h\dif x \in L_0^2(\Omega)$. 
Since $q_h$ is mean value free in $\Omega_h$, we have
\begin{equation}
  \begin{aligned}\label{bar_qh2Me}
  \bar q_h=q_h-\frac{1}{|\Omega|}\bigg(\int_{\Omega}q_h\dif x-\int_{\Omega_h}q_h\dif x\bigg)=q_h+(-1)^{\sigma}\frac{1}{|\Omega|}\sum_{e\in\E_h^b}\int_{\Omega_h^e}q_h\dif x.
  \end{aligned}
  \end{equation}
Similar to \cite{Yiliu_WG}, one gets
\begin{equation}
\begin{aligned}\label{bar_qh}
\|\bar q_h\|_{0,\Omega}\lesssim\|q_h\|_{0,\Omega_h},\qquad (\bar q_h,q_h)_{\Omega_h}=\|q_h\|_{0,\Omega_h}^2.\\
\end{aligned}
\end{equation}
It is well-known (see, e.g., \cite{FEM_Brenner}) that there exists a $\bl v\in \bl H_0^1(\Omega)$ satisfying $\mathrm{div}\,\bl v=-\bar q_h$ on $\Omega$
and $\|\bl v\|_{1,\Omega}\lesssim \|\bar q_h\|_{0,\Omega}\lesssim \|q_h\|_{0,\Omega_h}$,
where the hidden constant in $\lesssim$ may depend on the shape of $\Omega$ but not on $h$.
Extend $\bl v$ to $\bbR^2$ by setting its value to be $0$ outside $\Omega$. The extension is still denoted by $\bl v$, which shall not be mistaken for the extension $E\bl v = \bl v^E$ defined in Lemma \ref{lem:extension}. 
It is clear that, after the zero-extension, one has $\bl v\in H_0^1(\Omega\cup\Omega)$ and hence the interpolation $I_h \bl v \in V_h$ is well-defined.
By the commutative property (\ref{comm_diag_property}), the definition of $I_h$ and (\ref{bar_qh}), one has
\begin{equation}
\begin{aligned}\label{Inequality:b_h1}
b_{h1}(I_h\bl v,q_h)&=-(\mathrm{div}\,I_h\bl v,q_h)_{\Omega_h}+\sum_{e\in\E_h^b}\lal I_h\bl v\cdot\bl\n_h,q_h\ral_e\\
&=-(\Pi_{k-1}^{0}\mathrm{div}\,\bl v,q_h)_{\Omega_h}+\sum_{e\in\E_h^b}\lal \bl v\cdot\bl\n_h,q_h\ral_e\\
&\geq-(\mathrm{div}\,\bl v,q_h)_{\Omega_h}-\sum_{e\in\E_h^b}h^{-1/2}\| \bl v\cdot\bl\n_h\|_{0,e}h^{1/2}\|q_h\|_{0,e}\\
&\geq-(\mathrm{div}\,\bl v,q_h)_{\Omega_h}-\bigg(\sum_{e\in\E_h^b}h^{-1}\| \bl v\cdot\bl\n_h\|_{0,e}^2\bigg)^{1/2}\bigg(\sum_{e\in\E_h^b}h\|q_h\|_{0,e}^2\bigg)^{1/2}\\
&\geq-\bigg((\mathrm{div}\,\bl v,q_h)_{\Omega}+(\mathrm{div}\,\bl v,q_h)_{\Omega\backslash \Omega_h}\bigg)\\
&\quad -\bigg(\sum_{e\in\E_h^b}h^{-1}\| \bl v\cdot\bl\n_h\|_{0,e}^2\bigg)^{1/2}\bigg(\sum_{e\in\E_h^b}h\|q_h\|_{0,e}^2\bigg)^{1/2}.
\end{aligned}
\end{equation}

We then estimate the right-hand side of \eqref{Inequality:b_h1} one by one.
Note that $\mathrm{div}\,\bl v=-\bar q_h$ on $\Omega$. By (\ref{bar_qh}), (\ref{bar_qh2Me}) and Lemma \ref{lem:Omega_he2K}, one has
\begin{equation*}
  \begin{aligned}
  -(\mathrm{div}\,\bl v,q_h)_{\Omega}& = (\bar q_h,q_h)_{\Omega}=(\bar q_h,q_h)_{\Omega_h}-(\bar q_h,q_h)_{\Omega_h\backslash \Omega}+(\bar q_h,q_h)_{\Omega\backslash \Omega_h}\\
&=\|q_h\|_{0,\Omega_h}^2-(-1)^{\sigma}\sum_{e\in\E_h^b}\int_{\Omega_h^e}\bar{q}_h\, q_h\dif x\\
&\geq \|q_h\|_{0,\Omega_h}^2 -\sum_{e\in\E_h^b}\|q_h\|_{0,\Omega_h^e}^2-\frac{1}{|\Omega|}\bigg(\sum_{e\in\E_h^b}\int_{\Omega_h^e}q_h\dif x\bigg)^2\\
& \geq \|q_h\|_{0,\Omega_h}^2 - h\|q_h\|_{0,\Omega_h}^2-\frac{1}{|\Omega|}\sum_{e\in\E_h^b}|\Omega_h^e|\|q_h\|_{0,\Omega_h^e}^2\\
& \geq (1-O(h))\|q_h\|_{0,\Omega_h}^2.
\end{aligned}
  \end{equation*}

Using lemmas \ref{lem:extension} and \ref{lem:Omega_he2K}, we have
\begin{equation*}
\begin{aligned}
(\mathrm{div}\,\bl v,q_h)_{\Omega\backslash \Omega_h}&\leq\sum_{e\in\E_h^b}\|\mathrm{div}\,\bl v\|_{0,\Omega_h^e}\|q_h\|_{0,\Omega_h^e}\leq \|\mathrm{div}\,\bl v\|_{0,\Omega}\bigg(\sum_{e\in\E_h^b}\|q_h\|_{0,\Omega_h^e}^2\bigg)^{1/2}\\
&\lesssim h^{1/2} \|\bl v\|_{1,\Omega}\|q_h\|_{0,\Omega_h}\lesssim h^{1/2}\|q_h\|_{0,\Omega_h}^2.
\end{aligned}
\end{equation*}

By the trace-inverse inequality in Lemma \ref{lem:trace_inverse}, one gets
\begin{equation*}
\begin{aligned}
\bigg(\sum_{e\in\E_h^b}h\|q_h\|_{0,e}^2\bigg)^{1/2}\lesssim \|q_h\|_{0,\Omega_h}.
\end{aligned}
\end{equation*}
Note that $\bl v|_{\Gamma}=0$ and $\bl v=0$ on $\Omega_h\backslash \Omega$. By (\ref{BK_e2Omeg_e}) and Lemma \ref{lem:extension}, one has
\begin{equation*}
\begin{aligned}
\bigg(\sum_{e\in\E_h^b}h^{-1}\| \bl v\cdot\bl\n_h\|_{0,e}^2\bigg)^{1/2}
&\lesssim\bigg(\sum_{e\in\E_h^b}h^{-1}\| \bl v\|_{0,e}^2\bigg)^{1/2}\lesssim h^{1/2}\bigg(\sum_{e\in\E_h^b}\|\bl v\|_{1,\Omega_h^e}^2\bigg)^{1/2}\\
&\lesssim h^{1/2}\|\bl v\|_{1,\Omega}\lesssim h^{1/2}\|q_h\|_{0,\Omega_h}.
\end{aligned}
\end{equation*}

Combining the above, for $h$ sufficiently small, one gets
\begin{equation*}
\begin{aligned}
b_{h1}(I_h\bl v,q_h) \geq (1-O(h^{1/2}))\|q_h\|_{0,\Omega_h}^2\gtrsim \|q_h\|_{0,\Omega_h}^2.
\end{aligned}
\end{equation*}

A similar argument gives
$$b_{h0}(I_h\bl v,q_h) \gtrsim \|q_h\|_{0,\Omega_h}^2.$$

Then, for $i=0,1$, one obtains 
\begin{equation*}
\begin{aligned}
\sup_{\bl v_h\in V_h}\frac{b_{hi}(\bl v_h,q_h)}{\|\bl v_h\|_{0,h}}
\geq \frac{b_{hi}(I_h \bl v,q_h)}{\|I_h \bl v\|_{0,h}}
\gtrsim \frac{\|q_h\|_{0,\Omega_h}^2}{\|I_h \bl v\|_{0,h}},
\end{aligned}
\end{equation*}
where $I_h \bl v$ is the interpolation of the zero-extension of $\bl v\in \bl H_0^1(\Omega)$, as explained before.

Finally, we prove $\|I_h\bl v\|_{0,h}\lesssim \|q_h\|_{0,\Omega_h}$, which combined with the above inequality will give the desired inf-sup condition.
By Lemma \ref{Ih_sta_0norm} and the property (\ref{comm_diag_property}), we have
$$\|I_h\bl v\|_{0,K}\lesssim \|\bl v\|_{1,K},$$
and
$$\|\mathrm{div}\,I_h\bl v\|_{0,K}=\|\Pi_{k-1}^{0,K}\mathrm{div}\,\bl v\|_{0,K}\leq \|\mathrm{div}\,\bl v\|_{0,K}\lesssim |\bl v|_{1,K}.$$
By Lemma \ref{lem:T1_bounded}, \ref{Ih_sta_0norm} and definitions of $T^m$ and $T_1^m$,  it follows that
\begin{equation*}
\begin{aligned}
\|h_K^{-1/2}T^m(I_h\bl v)\cdot\wt{\bl\n}\|_{0,e}&\leq \|h_K^{-1/2}I_h\bl v\cdot\wt{\bl\n}\|_{0,e}+\|h_K^{-1/2}T_1^m(I_h\bl v)\cdot\wt{\bl\n}\|_{0,e}\\
&\lesssim \|h_K^{-1/2}I_h\bl v\cdot\wt{\bl\n}\|_{0,e}+\|I_h\bl v\|_{0,K}\\
&\lesssim \|h_K^{-1/2}I_h\bl v\|_{0,e}+\|\bl v\|_{1,K}.
\end{aligned}
\end{equation*}
By (\ref{BK_eh2e}) and Lemma \ref{lem:Omega_he2K} and Inequality (\ref{Ih_sta_1norm}), it holds
\begin{equation*}
\begin{aligned}
\|h_K^{-1/2}I_h\bl v\|_{0,e}&\lesssim h^{-1/2}(\|I_h\bl v\|_{0,\wt e}+h|I_h\bl v|_{1,\Omega_h^e})\\
&\lesssim h^{-1/2}(\|I_h\bl v\|_{0,\wt e}+h^{3/2}|\bl v|_{1,K}).
\end{aligned}
\end{equation*}
Note that $\bl v\in\bl H_0^1(\Omega)$, using lemmas \ref{lem:TRACE}, \ref{lem:Omega_he2K} and inequalities (\ref{BK_e2eh}), (\ref{Ih_app}), one gets
\begin{equation*}
\begin{aligned}
h_K^{-1/2}\|I_h\bl v\|_{0,\wt e}&=h_K^{-1/2}\|I_h\bl v-\bl v\|_{0,\wt e}\\
&\lesssim h^{-1/2}(\|I_h\bl v-\bl v\|_{0,e}+h|I_h\bl v-\bl v|_{1,\Omega_h^e})\\
&\lesssim h^{-1}\|I_h\bl v-\bl v\|_{0, K}+|I_h\bl v-\bl v|_{1, K} +h^{1/2} |I_h\bl v|_{1,\Omega_h^e}+h^{1/2}|\bl v|_{1,\Omega_h^e}\\
&\lesssim |\bl v|_{1,K\cup \Omega_h^e}.
\end{aligned}
\end{equation*}
Combining the above, we obtain $\|I_h\bl v\|_{0,h}\lesssim \|\bl v\|_{1,\Omega}\lesssim \|q_h\|_{0,\Omega_h}$. This completes the proof of the lemma.
\end{proof}

From Theorem 3.1 in \cite{well_posdeness_Nicolaides} (also see \cite{book:MFEM_Boffi}), lemmas \ref{lem:bounded} and \ref{lem:LBB} immediately imply that
\begin{thm} 
The discrete problem (\ref{Wh1}) is well-posed, i.e. it admits a unique solution and
\begin{equation}
\begin{aligned}\label{B_h_bound}
\|\bl\sigma_h,\zeta_h\|_H\lesssim \sup_{(\bl v_h, q_h)\in V_h\times Q_{0h}}\frac{B_h((\bl\sigma_h,\zeta_h),(\bl v_h, q_h))}{\|\bl v_h, q_h\|_H},
\qquad \forall\,(\bl\sigma_h,\zeta_h)\in V_h\times Q_{0h},
\end{aligned}
\end{equation}
where
\begin{equation*}
\begin{aligned}
B_h((\bl\sigma_h,\zeta_h),(\bl v_h, q_h))&:=a_h(\bl\sigma_h,\bl v_h)+b_{h1}(\bl v_h,\zeta_h)+b_{h0}(\bl\sigma_h,q_h),\\
\|\bl\sigma_h,\zeta_h\|_H&:=(\|\bl\sigma_h\|_{0,h}^2+\|\zeta_h\|_{0,\Omega_h}^2)^{1/2}.
\end{aligned}
\end{equation*}
\end{thm}

\subsection{Compatibility of (\ref{Wh1}) and an equivalent discrete form}
Recall that a compatibility condition (\ref{g_comp_cond}) is needed for the well-posedness of the continuous problem (\ref{primal_problem}), as well as its weak formulation \eqref{Weak_P}. We point out that the compatibility mechanism for the discrete problem works differently. In fact, the essential boundary condition $\bl u\cdot\bl\n=g_N$ on $\Gamma$ has been weakened in \eqref{Wh1}, and thus no relation is required between $\wt{g}_N$ and $f^E$. In Section \ref{subsec:wp}, we have proved that the discrete system \eqref{Wh1} admits a unique solution, and hence is ``compatible".

However, in practical implementation, the space $Q_{0h}$ is not commonly used. It is more convenient to construct a set of basis for $Q_h$ than for $Q_{0h}$. Hence we introduce an equivalent discrete problem: {\it Find $(\bl u_h,p_h)\in V_h\times Q_{h}$ such that}
\begin{equation}
\left\{
\begin{aligned}\label{Wh1a}
a_h(\bl u_h,\bl v_h)+b_{h1}(\bl v_h,p_h)&=l_h(\bl v_h), \quad &\forall&\bl v_h\in V_h,\\
b_{h0}(\bl u_h,q_h)+\sum_{e\in\E_h^b}\lal\bl u_h\cdot\bl\n_h, \bar q_h\ral_e&=-(f^E - \overline{f^E},q_h)_{\Omega_h},  \quad &\forall &q_h\in Q_{h},
\end{aligned}
\right.
\end{equation}
where $\bar q_h=\frac{1}{|\Omega_h|}\int_{\Omega_h} q_h\dif x$ and $\overline {f^E}=\frac{1}{|\Omega_h|}\int_{\Omega_h} f^E\dif x$.

\begin{lem}
    The systems \eqref{Wh1} and \eqref{Wh1a} are equivalent in the sense that they admit the same solution in $V_h\times Q_{0h}$.
\end{lem}
\begin{proof}
    The space $Q_{h}$ can be decomposed into $Q_{0h}\oplus \bbR$,
    and the decomposition is orthogonal in the $L^2(\Omega_h)$ norm. Using integration by parts, it is obvious that $b_{h1}(\bl v_h,1)=0$ for all $\bl v_h \in V_h$. Since $q_h-\bar{q}_h\in Q_{0h}$, the second equation of \eqref{Wh1} implies that, for all $q_h\in Q_h$,
    $$
    \begin{aligned}
     b_{h0}(\bl u_h,q_h - \bar{q}_h)&=-(f^E,q_h - \bar{q}_h)_{\Omega_h}, \\[2mm]
    \Rightarrow \quad  b_{h0}(\bl u_h,q_h)+\sum_{e\in\E_h^b}\lal\bl u_h\cdot\bl\n_h, \bar q_h\ral_e &= -(f^E,q_h - \bar{q}_h)_{\Omega_h},
    \end{aligned}
    $$
    which is just the second equation of \eqref{Wh1a}. 
\end{proof}

As long as \eqref{Wh1} is well-posed, problem \eqref{Wh1a} is solvable, i.e., the underlying discrete linear system is ``compatible". However, its solution is not unique since $(\bl u_h, \,p_h+constant)$ is also a solution. There are various ways to deal with this case, for example, using a Krylov subspace iterative solver, or adding a constraint to the linear system to ensure that the solution stays in $V_h\times Q_{0h}$.


\section{Error analysis}\label{sec5:error}

We first obtain an error equation by subtracting equations (\ref{eq:weak_1st})-\eqref{weak_2nd} from the discrete problem (\ref{Wh1})
\begin{equation}
  \begin{aligned}\label{error equation}
    E_R:=&~a_h(\bl u_h-\bl u^E,\bl v_h)+b_{h1}(\bl v_h,p_h-p^E)+b_{h0}(\bl u_h-\bl u^E,q_h) \\
    =&~-(\bl u^E+\nabla p^E,\bl v_h)_{\Omega_h}-(\mathrm{div}\,\bl u^E-f^E,\mathrm{div}\, \bl v_h)_{\Omega_h}
  +(\mathrm{div}\,\bl u^E-f^E, q_h)_{\Omega_h}\\
  &\quad+\sum_{e\in\E_h^b}h_K^{-1}\lal\wt{g}_N-T^m\bl u^E\cdot\wt{\bl\n}, T^m\bl v_h\cdot\wt{\bl\n}\ral_e.
\end{aligned}
\end{equation}

\begin{lem}\label{lem:ER_1}
  (cf. \cite{Burman_boundary_value_correction_2018}). Assume that $\bl u \in \bl H^{\max\{l+1,m+2\}}(\Omega),\,p\in H^{l+1}(\Omega)$ and $f\in H^l(\Omega)$ satisfy  Equation (\ref{primal_problem}), then 
  \begin{align}
  \|\bl u^E+\nabla p^E\|_{0,\Omega_h}&\lesssim \delta^{l}\|D^l(\bl u^E+\nabla p^E)\|_{0,\Omega_h\backslash\Omega},\label{R1}\\
  \|\mathrm{div}\,\bl u^E-f^E\|_{0,\Omega_h}&\lesssim\delta^{l}\|D^l(\mathrm{div}\,\bl u^E-f^E)\|_{0,\Omega_h\backslash \Omega},\label{R2}\\
  \sum_{e\in\E_h^b}h_K^{-1}\|\wt{g}_N-T^m\bl u^E\cdot\wt{\bl\n}\|_{0,e}^2&\lesssim \delta^{2m+2}h^{-1}\|D^{m+2}\bl u\|_{0,\Omega}^2.\label{R3}
\end{align}
\end{lem}
\begin{proof}
Inequalities (\ref{R1}) and (\ref{R2}) follow immediately from (3.14) in \cite{Burman_boundary_value_correction_2018}. The estimate of the residual (\ref{R3}) follows from (3.12) in \cite{Burman_boundary_value_correction_2018} and Lemma \ref{lem:extension}.
\end{proof}

Combining (\ref{error equation}) and using Lemma \ref{lem:ER_1}, one gets an estimate of the consistency error:

\begin{lem}\label{lem:consistency_err} (Consistency Error).
Assume that $\bl u \in \bl H^{\max\{l+1,m+2\}}(\Omega),\,p\in H^{l+1}(\Omega)$ and $f\in H^l(\Omega)$ satisfy  Equation (\ref{primal_problem}), we have
\begin{align*}
|E_R |\lesssim \bigg[\delta^{l}\bigg(\|D^l(\bl u^E+\nabla p^E)\|_{0,\Omega_h\backslash\Omega}&+\|D^l(\mathrm{div}\,\bl u^E-f^E)\|_{0,\Omega_h\backslash \Omega}\bigg) \\ & \qquad+\delta^{m+1}h^{-1/2}\|D^{m+2}\bl u\|_{0,\Omega}\bigg]\|\bl v_h,q_h\|_H.
\end{align*}
\end{lem}

\begin{lem}\label{lem:app_u-uI} (Interpolation Error).
Let $(\bl v,p)\in \bl H^{r+1}(\Omega)\times H^{t}(\Omega)$ with $\max\{m,1\}\leq r+1\leq k+1$ and $0\leq t\leq k$.  
Let $\bl v_I = I_h \bl v^E \in V_{h}$ be the interpolation of $\bl v^E$, and $p_I\in Q_{0h}$ be the $L^2$ orthogonal projection of $p^E$. Then
\begin{equation*}
\begin{aligned}
\|\bl v^E-\bl v_I\|_{0,h}+\|p^E-p_I\|_{0,\Omega_h}\lesssim h^{\min\{r,t,k\}}(|\bl v|_{r,\Omega}+|\bl v|_{r+1,\Omega}+|p|_{t,\Omega}).
\end{aligned}
\end{equation*}
\end{lem}

\begin{proof}
By Lemmas \ref{lem:Pi_Ih_app} and \ref{lem:extension}, we obtain
\begin{equation*}
\begin{aligned}
\|\bl v^E-\bl v_I\|_{0,\Omega_h}&\lesssim h^r|\bl v^E|_{r,\Omega_h}\lesssim h^r|\bl v|_{r,\Omega},\\
\|\mathrm{div}(\bl v^E-\bl v_I)\|_{0,\Omega_h}&\lesssim h^r|\mathrm{div}\,\bl v^E|_{r,\Omega_h}\lesssim h^r|\bl v^E|_{r+1,\Omega_h}\lesssim h^r|\bl v|_{r+1,\Omega},\\
\|p^E-p_I\|_{0,\Omega_h}&\lesssim h^t|p^E|_{t,\Omega_h}\lesssim h^t|p|_{t,\Omega},
\end{aligned}
\end{equation*} 
By Lemma \ref{lem:TRACE}, inequality (\ref{deta}) and interpolation error (\ref{Ih_app}), one has
\begin{equation*}
\begin{aligned}
\|h_K^{-1/2}T^m(\bl v^E-\bl v_I)\cdot\wt{\bl \n}\|_{0,e}&\lesssim \sum_{j=0}^m\delta_h^j h_K^{-1/2}(h_K^{-1/2}|\bl v^E-\bl v_I|_{j,K}+h_K^{1/2}|\bl v^E-\bl v_I|_{j+1,K})\\
&\lesssim \sum_{j=0}^m\delta_h^j h_K^{-1}h^{r+1-j}|\bl v^E|_{r+1,K}\\
&\lesssim \sum_{j=0}^m\bigg(\frac{\delta_h}{h}\bigg)^j h^{r-j}|\bl v^E|_{r+1,K}\\
&\lesssim  h^r|\bl v^E|_{r+1,K}.
\end{aligned}
\end{equation*}
Combining the above and using the definitions of norms, we complete the proof.
\end{proof}

\begin{lem}\label{lem:uI-uh}
  Let $(\bl u,p)\in \bl H^{\max\{r+1,l+1,m+2\}}(\Omega)\times H^{\max\{t,l+1\}}(\Omega)$ be the solution to Problem (\ref{primal_problem}), $f\in H^{l}(\Omega)$,
  and $(\bl u_h,p_h)\in V_h\times Q_{0h}$ be the discrete solution of (\ref{Wh1}). Then
\begin{equation*}
\begin{aligned}
&\|\bl u_h-\bl u_I\|_{0,h}+\|p_h-p_I\|_{0,\Omega_h}\\
\lesssim &~h^{\min\{r,t,k\}}(|\bl u|_{r,\Omega}+|\bl u|_{r+1,\Omega}+|p|_{t,\Omega})
+\delta^{m+1}h^{-1/2}\|D^{m+2}\bl u\|_{0,\Omega}\\
&\quad +\delta^{l}\bigg(\|D^l(\bl u^E+\nabla p^E)\|_{0,\Omega_h\backslash\Omega}+\|D^l(\mathrm{div}\,\bl u^E-f^E)\|_{0,\Omega_h\backslash \Omega}\bigg).
\end{aligned}
\end{equation*}
where $\bl u_I = I_h \bl u^E \in V_{h}$ is the interpolation of $\bl u^E$, and $p_I\in Q_{0h}$ is the $L^2$ orthogonal projection of $p^E$
\end{lem}

\begin{proof}
From (\ref{B_h_bound}), one has
\begin{equation*}
\begin{aligned}
\|\bl u_h-\bl u_I,p_h-p_I\|_H\lesssim \sup_{(\bl v_h, q_h)\in V_h\times Q_{0h}}\frac{B_h((\bl u_h-\bl u_I,p_h-p_I),(\bl v_h, q_h))}{\|\bl v_h, q_h\|_H}.
\end{aligned}
\end{equation*}
Note that
\begin{equation*}
\begin{aligned}
B_h((\bl u_h-\bl u_I,p_h-p_I),(\bl v_h, q_h))=&~a_h(\bl u_h-\bl u^E,\bl v_h)+b_{h1}(\bl v_h,p_h-p^E)+b_{h0}(\bl u_h-\bl u^E,q_h)\\
&+a_h(\bl u^E-\bl u_I,\bl v_h)+b_{h1}(\bl v_h,p^E-p_I)+b_{h0}(\bl u^E-\bl u_I,q_h)\\
=: &~ E_R+E_h.
\end{aligned}
\end{equation*}
By the definition of $\bl u_I$ and $p_I$, one gets
\begin{equation*}
\begin{aligned}
b_{h0}(\bl u^E-\bl u_I,q_h)&=(\mathrm{div}\,(\bl u_I-\bl u^E),q_h)_{\Omega_h}\\
&=\sum_{K\in\T_h}(\bl u^E-\bl u_I,\nabla q_h)_K-\sum_{K\in\T_h}\lal(\bl u^E-\bl u_I)\cdot\bl\n_h, q_h\ral_{\pt K}=0,
\end{aligned}
\end{equation*}
and
\begin{equation*}
\begin{aligned}
b_{h1}(\bl v_h,p^E-p_I)&=b_{h0}(\bl v_h,p^E-p_I)+\sum_{e\in\E_h^b}\lal\bl v_h\cdot\bl\n_h, p^E-p_I\ral_e\\
&=\sum_{e\in\E_h^b}\lal\bl v_h\cdot\bl\n_h, p^E-p_I\ral_e.
\end{aligned}
\end{equation*}
Therefore
\begin{equation*}
\begin{aligned}\label{err_5}
E_h&= a_h(\bl u^E-\bl u_I,\bl v_h)+b_{h1}(\bl v_h,p^E-p_I)+b_{h0}(\bl u^E-\bl u_I,q_h)\\
&=a_h(\bl u^E-\bl u_I,\bl v_h)+\sum_{e\in\E_h^b}\lal\bl v_h\cdot\bl\n_h, p^E-p_I\ral_e.
\end{aligned}
\end{equation*}
By the Cauchy-Schwarz inequality, lemmas \ref{lem:app_u-uI},\,\ref{lem:vn_e2h},\,\ref{lem:TRACE},\,\ref{lem:extension} and Inequality (\ref{pi_app}), one gets
\begin{equation*}
\begin{aligned}
E_h&\leq \|\bl u^E-\bl u_I\|_{0,h} \|\bl v_h\|_{0,h} +\sum_{e\in\E_h^b}h_K^{-1/2}\|\bl v_h\cdot\bl\n_h\|_{0,e}h_K^{1/2}\|p^E-p_I\|_{0,e}\\
&\lesssim (\|\bl u^E-\bl u_I\|_{0,h} +\|p^E-p_I\|_{0,\Omega_h}+h|p^E-p_I|_{1,\Omega_h}) \|\bl v_h\|_{0,h}\\
&\lesssim h^{\min\{r,t,k\}}(|\bl u|_{r,\Omega}+|\bl u|_{r+1,\Omega}+|p|_{t,\Omega})\|\bl v_h\|_{0,h}.
\end{aligned}
\end{equation*}
Combining the above estimate of $E_h$ above and the estimate of $E_R$ in Lemma \ref{lem:consistency_err}, we obtain the proof of the lemma.
\end{proof}

Using the above lemmas and the triangle inequality, we get the main theorem of this section:

\begin{thm}\label{thm:L2err}
Under the same assumption of Lemma \ref{lem:uI-uh}, one has
\begin{equation}
\begin{aligned}\label{arbitrary_m}
&\|\bl u^E-\bl u_h\|_{0,h}+\|p^E-p_h\|_{0,\Omega_h}\\
\lesssim &~h^{\min\{r,t,k\}}(|\bl u|_{r,\Omega}+|\bl u|_{r+1,\Omega}+|p|_{t,\Omega})
 +\delta^{m+1}h^{-1/2}\|D^{m+2}\bl u\|_{0,\Omega} \\
&\qquad +\delta^{l}\bigg(\|D^l(\bl u^E+\nabla p^E)\|_{0,\Omega_h\backslash\Omega}+\|D^l(\mathrm{div}\,\bl u^E-f^E)\|_{0,\Omega_h\backslash \Omega}\bigg).
\end{aligned}
\end{equation}
Moreover, the terms $\|D^l(\bl u^E+\nabla p^E)\|_{0,\Omega_h\backslash\Omega}$ and $\|D^l(\mathrm{div}\,\bl u^E-f^E)\|_{0,\Omega_h\backslash \Omega}$ vanish when $\Omega$ is convex.
\end{thm}

\begin{re}\label{re:an_important_trick}
(An implementation trick). When $m=k$, one has $T^m \bl v_h = \bl v_h$ for all $\bl v_h\in V_h$, which greatly simplifies the implementation as one no longer needs to compute the Taylor expansion. 
\end{re}

For a given polynomial order $k\ge 1$, we want to find when the proposed discretization \eqref{Wh1} reaches the optimal convergence rate $O(h^k)$. 
Suppose that Inequality \eqref{deta} holds, that is,  $\delta\lesssim h^2$.
Then Theorem \ref{thm:L2err} indicates that one has to have
$$
r\ge k,\quad t\ge k,\quad 2l\ge k,\quad 2m+3/2\ge k, 
$$
to obtain $O(h^k)$ convergence. By taking $r=t=k$ and $l=\max\{1,k-1\}$, we get the following corollary:

\begin{cro} \label{cor:error}
Assume \eqref{deta} holds, $\max\{0, k/2 - 3/4\} \le m\le k$, and $l=\max\{1,k-1\}$. 
  Let $(\bl u,p)\in \bl H^{k+2}(\Omega)\times H^{l+1}(\Omega)$ be the solution to Problem (\ref{primal_problem}), $f\in H^{l}(\Omega)$,
  and $(\bl u_h,p_h)\in V_h\times Q_{0h}$ be the discrete solution of (\ref{Wh1}), then
\begin{equation} \label{L2err_pi}
\|\bl u^E-\bl u_h\|_{0,h}+\| p^E-p_h\|_{0,\Omega_h} \lesssim h^k \left( \|\bl u\|_{k+2,\Omega} + \|p\|_{l+1,\Omega} + \|f\|_{l,\Omega} \right).
\end{equation}
  When $\Omega$ is convex, the above result holds for $l=k-1$.
\end{cro}

\section{Numerical experiment}\label{sec7:numerical}
In this section, we present numerical results for problem \eqref{primal_problem} on two types of curved domains, as shown in Fig. \ref{meshes}. Body-fitted, unstructured triangular meshes are used in the computation. We always set $m=k$, which simplifies the implementation according to Remark \ref{re:an_important_trick}. For convenience, denote
$$E_{u,p}=\|\bl u^E-\bl u_h\|_{0,h}+\|p^E-p_h\|_{0,\Omega_h}.$$
If no boundary correction technique is used, then $\|\cdot\|_{0,h}$ is just the $H(\textrm{div},\Omega_h)$ norm.

\begin{figure}[!htbp]
\centering
\subfigure[]
{
\includegraphics[height=4cm,width=5.5cm]{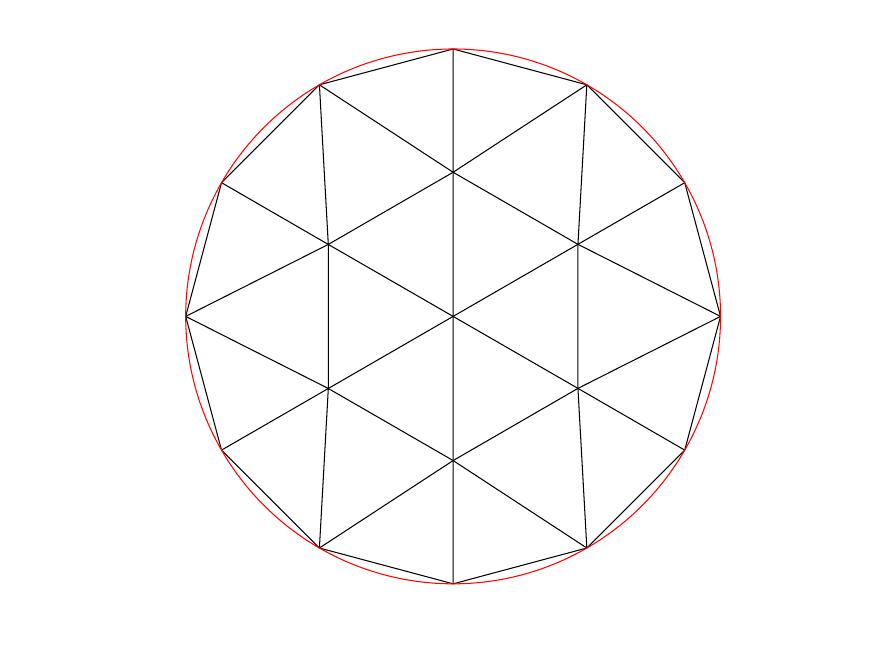}
}\qquad
\subfigure[]
{
\includegraphics[height=4cm,width=5.5cm]{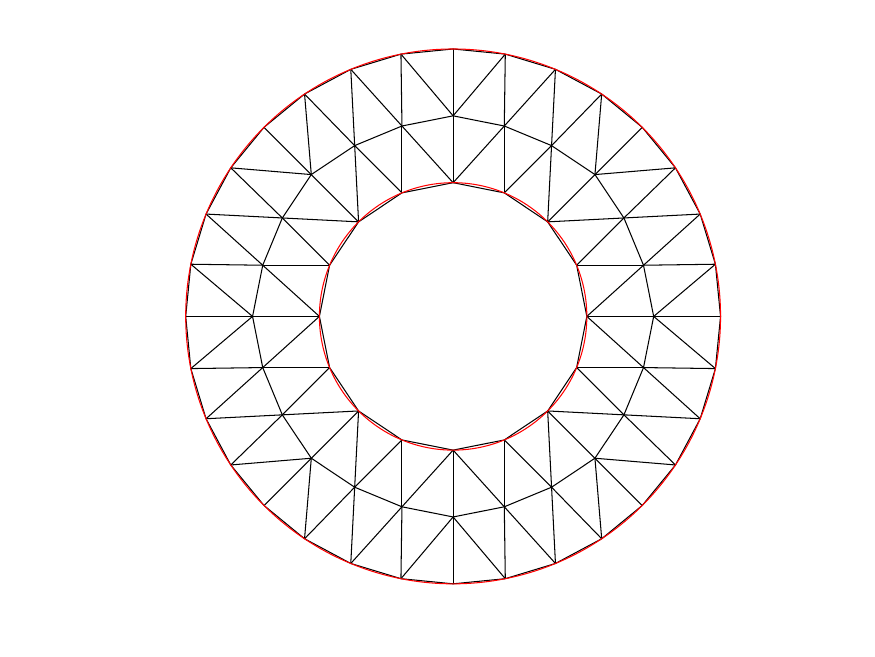}
}\qquad
\caption{(a). The mesh with $h=1/2$ on circular domain. (b). The mesh with $h=1/4$ on ring domain.}
\label{meshes}
\end{figure}

\vspace{0.2cm}
\noindent${\bf{Example ~1. ~Circular ~domain}}$. 
\vspace{0.2cm}

The domain is a unit disk $\{(x,y):x^2+y^2< 1\}$. The exact solution is
\begin{equation*}
\begin{aligned}
\bl u(x,y)=\dbinom{-3x^2-y^2+3}{-2xy},\qquad p(x,y)=-3x+x(x^2+y^2),
\end{aligned}
\end{equation*}
which satisfies a homogeneous Neumann boundary condition $\bl u\cdot\bl\n=0$ on $\Gamma$, and $\int_{\Omega}p\dif x=0$.
Note that $\bl u$, $p$ and $f$ are smooth, i.e. they satisfy the regularity requirement of Cor. \ref{cor:error}. We solve Example $1$ on meshes illustrated in Fig. \ref{meshes} (a), for various size of $h$.

We first test Example $1$ using the primal mixed formulation \cite{book:FE_RT} with $BDM_k$ discretization, without boundary correction. A homogeneous Neumann (essential) boundary condition is imposed on $\Gamma_h$. We report $E_{u,p}$
in Tab. \ref{tab_eg2_NCu}, which only reaches $O(h^{1/2})$ for $k=1,2,3$, i.e. there is a loss of accuracy as expected. 

\begin{table}[!htbp]
  \caption{\small\textit{Errors of Example 1 without boundary correction on circular domain.}}
  \label{tab_eg2_NCu}
  \begin{center}
  {\small
  \begin{tabular}{c | c| c| c|c|c|c}
  \hline
  \multirow{2}{*}{$h$}&\multicolumn{2}{c|}{$k=1$}&\multicolumn{2}{c|}{$k=2$} &\multicolumn{2}{c}{$k=3$} \\
  \cline{2-7}
            & $E_{u,p}$  & order  & $E_{u,p}$ & order & $E_{u,p}$  & order\\
  \hline
  1/8   &  4.89e-01   &   --    &   3.48e-01   &   --     &   3.44e-01   &   --\\
  \hline
  1/16  &  3.06e-01   &   0.68   &   2.53e-01   &   0.46   &   2.51e-01   &   0.45\\
  \hline
  1/32  &  1.96e-01   &   0.65   &   1.77e-01   &   0.52   &   1.76e-01   &   0.51\\
  \hline
  1/64   &  1.29e-01   &  0.60     &   1.22e-02   & 0.54     &   1.22e-02   &   0.53  \\
  \hline
  \end{tabular}}
  \end{center}
  \end{table}

We then solve the same problem using the proposed discretization (\ref{Wh1}) with boundary correction. 
We first set $m=k$ and present the result in Tab. \ref{tab_eg2_Cu}. An $O(h^{k})$ optimal convergence is observed, which agrees well with the theoretical result in Cor. \ref{cor:error}.

Cor. \ref{cor:error} also gives a lower bound $m\ge \max\{0,k/2-3/4\}$ in order for the scheme to reach optimal convergence.
For $k=1,2,3$, the smallest $m$ that satisfies the lower bound is $0,1,1$, respectively. We test the scheme (\ref{Wh1}) with these values, and present the results in Tab. \ref{tab_eg2_TCu}. 
Again, optimal $O(h^{k})$ convergence is observed, as predicted in Cor. \ref{cor:error}.

\begin{table}[!htbp]
  \caption{\small\textit{Errors of Example 1 with boundary correction for $m=k$ on circular domain.}}
  \label{tab_eg2_Cu}
  \begin{center}
  {\small
  \begin{tabular}{c | c| c| c|c|c|c}
  \hline
  \multirow{2}{*}{$h$}&\multicolumn{2}{c|}{$k=1$}&\multicolumn{2}{c|}{$k=2$} &\multicolumn{2}{c}{$k=3$} \\
  \cline{2-7}
            & $E_{u,p}$  & order  & $E_{u,p}$ & order & $E_{u,p}$  & order\\
  \hline
  1/8   &     3.84e-01   &   --    &   3.54e-03   &   --     &   6.54e-05      &   --   \\
  \hline
  1/16   &    1.94e-01   &   0.98  &   7.61e-04   &   2.22   &   7.45e-06      &   3.13   \\
  \hline
  1/32   &    9.47e-02   &   1.03  &   1.67e-04   &   2.18   &   7.85e-07      &   3.25   \\
  \hline
  1/64   &    4.64e-02   &   1.03  &   3.86e-05   &   2.12   &   8.73e-08      &   3.17   \\
  \hline
  \end{tabular}}
  \end{center}
\end{table}
\begin{table}[!htbp]
  \caption{\small\textit{Errors of Example 1 with boundary correction for suitable $m$ on circular domain.}}
  \label{tab_eg2_TCu}
  \begin{center}
  {\small
  \begin{tabular}{c|c|c|c|c|c|c}
  \hline
  \multirow{2}{*}{$h$}&\multicolumn{2}{c|}{$k=1,m=0$}&\multicolumn{2}{c|}{$k=2,m=1$}&\multicolumn{2}{c}{$k=3,m=1$} \\
  \cline{2-7}
            & $E_{u,p}$  & order  & $E_{u,p}$  & order& $E_{u,p}$  & order\\
  \hline
  1/8    &3.86e-01 & --   & 3.57e-03 & --   & 6.47e-05  &  --\\
  \hline
  1/16   &1.95e-01 & 0.99 & 7.71e-04 & 2.21 & 7.42e-06  &  3.13\\
  \hline
  1/32   &9.48e-02 & 1.04 & 1.69e-04 & 2.19 & 7.82e-07  &  3.25 \\
  \hline
  1/64   &4.65e-02 & 1.03 & 3.89e-05 & 2.12 & 8.68e-08  &  3.17\\
  \hline
  \end{tabular}}
  \end{center}
  \end{table}

\vspace{0.2cm} 
\noindent$\bf{Example~ 2. ~Ring ~domain.}$
\vspace{0.2cm}

We repeat the above tests in a ring domain $\{(x,y)|0.25< x^2+y^2< 1\}$, which is not convex. The exact solution is
\begin{equation*}
\begin{aligned}
\bl u(x,y)=\dbinom{2\pi \cos(2\pi x)\sin(2\pi y)}{ 2\pi\sin(2\pi x)\cos(2\pi y)},\qquad p(x,y)=-\sin(2\pi x)\sin(2\pi y),
\end{aligned}
\end{equation*}
which satisfies a non-homogeneous Neumann boundary condition $\bl u\cdot\bl\n\neq0$ on $\Gamma$ and $\int_{\Omega}p\dif x=0$.
This time, we did not perform the test without boundary correction, since there is no effective way to impose the nonhomogeneous Neumann boundary condition on $\Gamma_h$.

We test Example $2$ on meshes illustrated in Fig. \ref{meshes} (b), for various sizes of $h$, using the scheme (\ref{Wh1}) with boundary correction. 
The results are reported in Tabs. \ref{tab_eg2_Ru} and \ref{tab_eg2_TRu},
both indicating optimal convergence.
This verifies that Cor. \ref{cor:error} holds on nonconvex domains.

\begin{table}[!htbp]
\caption{\small\textit{Errors of Example 2 with boundary correction for $m=k$ on ring domain.}}
\label{tab_eg2_Ru}
\begin{center}
{\small
\begin{tabular}{c | c| c| c|c|c|c}
\hline
\multirow{2}{*}{$h$}&\multicolumn{2}{c|}{$k=1$}&\multicolumn{2}{c|}{$k=2$} &\multicolumn{2}{c}{$k=3$} \\
\cline{2-7}
    & $E_{u,p}$  & order  & $E_{u,p}$ & order & $E_{u,p}$  & order\\
\hline
1/8 & 1.36e+01   &   --     &   1.66e+00   &   --     &   1.50e-01   &   -- \\
\hline
1/16 & 6.86e+00  &   0.99   &   4.20e-01   &   1.98   &   1.89e-02   &   2.99\\
\hline
1/32& 3.44e+00   &   1.00   &   1.05e-01   &   2.00   &   2.37e-03   &   3.00\\
\hline
1/64& 1.72e+00   &   1.00   &   2.64e-02   &   2.00   &   2.96e-04   &   3.00 \\
\hline
\end{tabular}}
\end{center}
\end{table}

\begin{table}[!htbp]
  \caption{\small\textit{Errors of Example 2 with boundary correction for suitable $m$ on ring domain.}}
  \label{tab_eg2_TRu}
  \begin{center}
  {\small
  \begin{tabular}{c | c|c|c|c|c|c}
  \hline
  \multirow{2}{*}{$h$}&\multicolumn{2}{c|}{$k=1,m=0$}&\multicolumn{2}{c|}{$k=2,m=1$}&\multicolumn{2}{c}{$k=3,m=1$} \\
  \cline{2-7}
            & $E_{u,p}$  & order  & $E_{u,p}$ & order &  $E_{u,p}$  & order\\
   \hline
  1/8 &  1.36e+01  &   --  & 1.66e-00& --  &1.50e-01&--\\
  \hline
  1/16& 6.86e+00   &   0.99& 4.20e-01& 1.98&1.89e-02&2.99\\
  \hline
  1/32& 3.44e+00   &   1.00& 1.05e-01& 2.00&2.37e-03&3.00\\
  \hline
  1/64& 1.72e+00   &   1.00& 2.64e-02& 2.00&2.96e-04&3.00\\
  \hline
  \end{tabular}}
  \end{center}
  \end{table}

\section{Conclusion}\label{sec8:conclusion}
In this paper, we consider the mixed finite element method for second-order elliptic equations on domains with curved boundaries.
A boundary value correction technique is proposed to compensate for the loss of accuracy in the discretization. 
The proposed scheme achieves optimal convergence.
Numerical experiments further validate the theoretical results.

 \bibliographystyle{abbrv}

\end{document}